\documentclass[12pt]{amsart}
\usepackage{geometry} 
\usepackage[all]{xypic}

\newtheorem{thm}{Theorem}[section]
\newtheorem{cor}[thm]{Corollary}
\newtheorem{lem}[thm]{Lemma}
\newtheorem{prop}[thm]{Proposition}
\theoremstyle{definition}
\newtheorem{defn}[thm]{Definition}
\theoremstyle{remark}
\newtheorem{rem}[thm]{Remark}
\numberwithin{equation}{section}

\newcommand{\C}{\mathbb{C}} 
\newcommand{\R}{\mathbb{R}}

\DeclareMathOperator{\Sym}{Sym}
\DeclareMathOperator{\Herm}{Herm}
\DeclareMathOperator{\ci}{\sqrt{-1}}

\newcommand{\subsoln}{\underline{\mathcal{S}}}

\newcommand{\abs}[1]{|#1|}
\newcommand{\norm}[1]{|| #1||}
\newcommand{\dist}[1]{\mathrm{dist}(#1)}
\newcommand{\diam}[1]{\mathrm{diam}(#1)}
\DeclareMathOperator*{\osc}{osc}

\title{A Viscosity Approach to the Dirichlet Problem for Complex Monge-Amp\`ere Equations}
\author{Yu Wang }

\begin{document}

\begin{abstract}
The Dirichlet problem for complex Monge-Amp\'ere equations with continuous data 
is considered. In particular,
a notion of viscosity solutions is introduced; a comparison principle and a solvability theorem
are proved; the equivalence between viscosity and pluripotential solutions
is established; and an ABP-type of $L^{\infty}$-estimate is achieved.
\end{abstract}

\maketitle

\setcounter{section}{0}
\section{Introduction}

Viscosity methods provide a powerful tool for the study of non-linear partial
differential equations (see e.g. \cite{CIL, IL, CC} and references therein).
However, viscosity methods have been developed largely so far for
real equations, where no complex structure plays a particular role.
It is only relatively recently that the incorporation of
complex structures, such as plurisubharmonicity
in the case of the complex Monge-Amp\`ere equation,
has been considered by Harvey and Lawson \cite{HL} and particularly
by Eyssidieux, Guedj, and Zeriahi \cite{EGZ2}.

\smallskip
In the complex case, even the notion of viscosity solutions, especially
viscosity supersolutions, presents some subtleties. A detailed discussion
of this can be found in \cite{EGZ2}, who solved this problem
in the case of the complex Monge-Amp\`ere equation $(\omega_0+{i\over 2}\partial\bar\partial u)^n
=e^{u}\phi(z)\omega^n$, on a compact K\"ahler manifold $(X,\omega)$
with $\phi>0$  and $\omega_0\geq 0$.
Their approach is partly motivated by the desire
to apply existing results on generalized
solutions \cite{IL, CIL, BEGZ}, and by the specific objective
of establishing the continuity of the potentials for the singular K\"ahler-Einstein
metrics which they constructed earlier in \cite{EGZ1}. 
In \cite{HL}, Monge-Amp\`ere equations with constant right hand sides were considered.
But to the best of our knowledge,
a general and systematic viscosity treatment of general complex Monge-Amp\`ere equations
is not yet available in the literature.

\smallskip
The main goal of this paper is to develop such a treatment.
Specifically we consider the Dirichlet problem
\begin{equation}
\label{intro DP}
\begin{cases}
M_{\C} (u):=\det( 2 u_{\bar{k} i}  ) =  f(z, u)  & \text{in} \;\Omega \\
u  = g  & \text{on} \; \partial \Omega \\
u \in PSH (\Omega) \cap C(\overline{\Omega})
\end{cases}
\end{equation}
where $\Omega$ is a strictly pseudoconvex bounded domain in $\C^n$,
$PSH(\Omega)$ denotes the space of plurisubharmonic functions on $\Omega$,
and $f , g$ are given continuous functions with $f \geq 0$,
$f(z,u)$ non-decreasing in $u$. 
We note that $f$ is not required to be strictly increasing, so our discussion
includes automatically the case $f(z,u)=\phi(z)$.

\medskip
We formulate a notion of viscosity solutions for such equations (see Definition \ref{defn viscosity}
below). A new feature in our viscosity formulation is the decomposition of the space of testing functions. 
This is important for the invariance properties of the equations, and does not seem
to have been considered before.
Our definition is somewhat similar to that of \cite{Gu} and \cite{IL} for the real Monge-Amp\`ere
equation. 
However, the treatment of real Monge-Amp\`ere equations in \cite{IL} makes use of the local Lipschitz property of convex functions, and hence cannot be directly carried over.

\smallskip
A key ingredient for our approach is the following comparison theorem,
which allows a right-hand side $f(z,u)$ depending on $u$:

\begin{thm}
\label{intro CP}
Let $\Omega$ be a bounded domain in $\C^n$. Let $f \in C(\Omega \times \R)$ be non-negative and for any fixed $z\in \Omega$, $f(z, \cdot )$ is non-decreasing; Assume that $u,v \in C(\overline{\Omega}) \cap PSH (\Omega)$ are viscosity subsolution and supersolution of the equation $M_{\C} (\cdot) = f(z, \cdot)$ respectively. Then
\[
u \leq v \; \text{on} \;\partial \Omega \; \Rightarrow  \; u \leq v  \; \text{in} \; \Omega. 
\]
\end{thm}

\smallskip

The main idea in the proof of Theorem \ref{intro CP}
originates from the work of Caffarelli and Cabre \cite{CC}, and we make use
of convex envelope and contact set techniques.
Although they are known in the study of many non-divergence equations, including the real Monge-Amp\`ere equations \cite{C1}, \cite{C2}, these techniques do not appear to have
been  applied to the complex Monge-Amp\`ere equations.
An immediate advantage of these techniques is that we can give a unified approach to comparison principles without treating separately the cases where $f(z,u)$ has zeros or $f(z,u)$ does not increase in $u$ strictly.
We also exploit a simple but useful inequality between real and complex Monge-Amp\`ere equations. 
The use of this inequality goes back to Cheng and Yau (unpublished), and subsequent works of Bedford \cite{BT1}, Cegrell and Persson \cite{CL}, and Blocki \cite{Bl}. Combining with viscosity techniques, we can give a more transparent version
of the proof by Cheng and Yau (see p.75 \cite{BT1}) and obtain the following ABP-type of $L^{\infty}$-estimate, which improves the $L^2$-stability in \cite{CL}: 

\begin{thm}
\label{C0e ABP}
Let $\Omega$ be a domain such that $\Omega \subset \subset B_r \subset B_{2r}$ and $f \in C(\overline{\Omega})$ non-negative. Let $u \in C(\overline{\Omega})$ be a viscosity supersolution of the equation $M_{\C} (u) = f$ in $\Omega$ and $u \geq 0$ on $\partial \Omega$. Then
\begin{equation}
\label{C0 estimate}
\sup_{\Omega} u^{-}  \leq C(n) r  \norm{f \cdot \chi_{\{ u =\Gamma_u\}}}_{L^2(\Omega)}^{1/n},
\end{equation}
where $C$ only depends on $n$.
\end{thm}

\smallskip
The comparison principle (Theorem \ref{intro CP}) implies readily the following existence theorem
for viscosity solutions:

\begin{thm}
\label{intro solve}
Let $\Omega$ be a bounded domain in $\C^n$ and $g \in C(\partial \Omega)$. Assume that: i) $f  \in C(\overline{\Omega} \times \R)$ be non-negative and for any fixed $z\in \Omega$, $f(z, \cdot )$ is non-decreasing; ii) there exists a harmonic function $\overline{u}$ with boundary value $g$;. iii) there exists a viscosity subsolution $\underline{u} \in C(\overline{\Omega})$ of {\rm (\ref{intro DP})} such that $\underline{u} =g$ on $\partial \Omega$. Then there exists a unique viscosity solution of the Dirichlet problem {\rm (\ref{intro DP})}.
\end{thm}

We note that the interest in Theorem \ref{intro solve} resides in the fact that the data
is only continuous. For smooth data, the existence of smooth solutions has been established
directly by Guan \cite{GB}, using the method of continuity and building
on the works of Caffarelli, Kohn, Nirenberg, and Spruck \cite{CKNS}.

Our proof of Theorem \ref{intro solve} also differs from standard reference (\cite{CIL},\cite{IL}). We learned this argument in Prof. Ovidiu Savin's lectures at Columbia University. Besides proving the existence of the solution, this argument also shows that the modulus of continuity of solution is controlled by the given subsolution and given datum (See. Corollary \ref{mod}).

\smallskip
Finally, we relate viscosity solutions in our sense to generalized solutions
in the sense of pluripotential theory:

\begin{thm}
\label{intro eqv} 
Let $\Omega$ be a bounded domain in $\C^n$ and $u \in C(\overline{\Omega}) \cap PSH (\Omega)$. Then $M_{\C} (u ) = f(z )$ in viscosity sense if and only if $M_{\C} (u) = f(z)$ in pluripotential sense.
\end{thm}

An immediate consequence of Theorems \ref{intro solve} and \ref{intro eqv} is
the following Corollary, which can be viewed as a generalization
of the classic result of Bedford and Taylor for the case when $f(z,u)$
does not depend on $u$:

\begin{cor}
\label{intro cor}
Let $\Omega$ be a strictly pseudoconvex domain and $g \in C(\partial \Omega)$. Assume that $f  \in C(\overline{\Omega} \times \R)$ is non-negative and for any fixed $z\in \Omega$, $f(z, \cdot )$ is non-decreasing. Then there exists a unique pluripotential (equivalent, viscosity) solution of the Dirichlet problem {\rm (\ref{intro DP}) }.
\end{cor}

\medskip

The paper is organized as follows. In \S 1, we summarize some basic linear algebra 
which is not standard. In \S
2, we introduce our notion of viscosity solutions
and establish some of their basic properties.
In \S 3, we adapt some standard tools from viscosity theory.
\S 4 is devoted to the proof of the comparison principle
and the study of the Dirichlet problem for the complex Monge-Amp\`ere
equation in strictly pseudoconvex domain. The equivalence between pluripotential solutions and viscosity solutions is established in \S 5.
We prove Theorem \ref{C0e ABP} in \S 6.

\bigskip

\textbf{Acknowledgments}: I would like to express my great gratitude to my advisor Prof. Duong Hong Phong, for his penetrating remarks, various suggestions and many encouragement. Also I am heartily grateful to Prof. Ovidiu Savin, for his valuable and inspirational discussions. I would also like to thank Prof. Zbigniew Blocki for providing me with some important references.

This paper is part of my forthcoming thesis at Columbia University.

\section{Linear Algebra Preliminaries}

Let $\R^{2n}$ be equipped with canonical coordinates ordered
as $x^1, ... ,x^n$,  $y^1, ... y^n$. 
We identify $C^n$ with $\R^{2n}$ by the relation $x^i + \ci y^i  =z^i$, and 
we let $J$ be the canonical complex structure:
\begin{equation}
\label{complex structure}
J  =  \begin{pmatrix}
  0 & -I  \\
  I & 0  
 \end{pmatrix}, 
\end{equation}
where $I$ is the $n \times n$ identity matrix. We introduce the following terminology and notations:

\medskip

$\bullet$ $D^2$ is the Hessian operator computed with respect to the ordered coordinates of $\R^{2n}$.

$\bullet$ $dd^c$ is the operator sending a smooth function $\varphi$ to the Hermitian matrices $2\varphi_{\bar{j} i}$, computed with respected to the coordinate $z^i$.

$\bullet$ $T(n)$ is the space of polynomials of the form $A_{\bar{j}i} z^{i} \bar{z^j}$, where $A_{\bar{j} i}$ is a Hermitian matrices. $T^{+} (n)$ is the non-negative cone in $T(n)$ consisting of $p = A_{\bar{j}i} z^{i} \bar{z^j}$ with $A$ being non-negative. 

$\bullet$ $H(n):= \C[n](2) \oplus \C[n](1)$, where $\C[n] (d)$ denotes space of homogeneous polynomials over $\C$ of degree $d$ with respect to $z^1, ..., z^n$.

$\bullet$ $\Herm(n)$ and $\Sym(2n)$ are the spaces of Hermitian matrices over $\C$ and symmetric matrices over $\R$ respectively. Recall that each $M \in \Sym(2n)$ that commutes with $J$ can be canonically identified as an element in $\Herm(n)$.

$\bullet$ $\det_{\C}$ and $\det_{\R}$ are the determinant functions on $\Herm(n)$ and $\Sym(2n)$ respectively. For any smooth function $\varphi$, we define 
\[
M_{\C} (\varphi) : = \text{det}_{\C} (dd^c \varphi),
\qquad
M_{\R} (\varphi) := \text{det}_{\R} (D^2 \varphi). 
\]
\medskip

The following lemmas regarding quadratic polynomials and matrices are very elementary; we shall omit the proof.
\begin{lem}
\label{decompose}
Let $\varphi $ be a quadratic polynomial in  $\R[x^1, ... ,x^n, y^1, ... y^n]$. Then under the identification $x^i + \ci y^i  =z^i$, $\varphi$ can be uniquely written as:
\[
\varphi  = p + \Re(h) ,  \quad  p \in T(n),\quad h \in H(n) .
\]
\end{lem}

\begin{lem}
\label{RC inequality}
Let $\varphi$ be a quadratic polynomial (or a smooth function), then $\frac{1}{2} ( D^2 \varphi + J^{t} D^2 \varphi J ) $ can be canonically identified with $dd^c \varphi$, and
\[
\mathrm{det}_{\R } ( \frac{1}{2} ( D^2 \varphi + J^{t} D^2 \varphi J )  ) = ( \mathrm{det}_{\C}  (dd^c \varphi)  )^2.
\]
Moreover if $\varphi$ is convex, then 
\[ M_{\C}^{1/n} (\varphi) \geq M_{\R}^{1/2n} (\varphi ). \]
\end{lem}

\section{Viscosity Solutions}
We consider the following equation:
\begin{equation}
\label{Eq}
M_{\C} ( u) (z) = f(z, u (z))   \quad \forall z \in \Omega
\end{equation}
Henceforth, we shall always assume the right-hand side $f $ to be a function 
from $ \Omega \times\R$ to $[0,\infty)$ (not necessarily continuous) and non-decreasing with respect to the second variable.

\medskip

We recall the following terminology:

\begin{defn}
Let $u, \varphi$ be two continuous functions.

{\rm (a)} $\varphi$ is said to touch $u$ from above at $z$ in a open neighborhood $V$ if $\varphi \geq u$ in $V$ and $\varphi (z) =  u(z)$.

{\rm (b)} $\varphi$ is said to touch $u$ from above at $z$, if there exists some open neighborhood $V$ of $z$ such that $\varphi$ touches $u$ from above in $V$.

The notions of $\varphi$ touching $u$ from below at $x$ are defined in an analogous way.
\end{defn}

We can now define the notion of viscosity solutions of the equation $M_{\C} (\cdot ) = f(z, \cdot)$ as follows:  

\begin{defn}
\label{defn viscosity}
Let $u$ be a continuous plurisubharmonic (``psh") function on a domain $\Omega \subset \C^n$.
The function $u$ is said to be a viscosity subsolution (resp. supersolution) of $M_{\C} (\cdot)  =f (z,\cdot) $, if the following condition holds:

\smallskip

($\star$) For any $p \in  T^{+} (n)$, if there exists $h \in H(n)$ such that $  p + \Re (h) $ touches $u$ from above (resp. below) at some $z_0\in \Omega$, then $M_{\C} (p) \geq f (z_0, u (z_0))$ (resp. $M_{\C} (p) \leq f(z_0, u(z_0) )$.

\smallskip

We say that $M_{\C} (u) \geq f(z, u)$ (resp. $M_{\C} (u) \geq f (z, u)$ ) in viscosity sense if $u$ is a viscosity subsolution (resp. supersolution).
We say that
$u$ is a viscosity solution of $M_{\C} (\cdot)  = f (z,\cdot)$ if it is both a supersolution and a subsolution.
\end{defn}

\medskip

\begin{rem}

i) In the Definition \ref{defn viscosity}, $\Omega$ is not required to be strictly pseudoconvex or bounded.

ii) One can similarly define $M^{1/n}_{\C}$; and $M^{1/n}_{\C} (u) = f^{1/n}(z,u)$ (resp. $\geq, \leq$) in viscosity sense if and only if $M_{\C} (u) = f(z,u)$ (resp. $\geq, \leq$) in viscosity sense.

\end{rem}

\medskip

The following proposition is useful in exploring the generality of our definition:
\begin{prop}
\label{test condition}

{\rm (a)} $u$ is a viscosity subsolution (resp. supersolution) of $M_{\C} (\cdot)  = f (z,\cdot)$ according to Definition {\rm \ref{defn viscosity}} if and only if the following condition holds:

\smallskip

For any $\varphi \in C^2 (\Omega) \cap PSH (\Omega)$, if $u - \varphi$ takes a local maximum (resp. minimum) at $z_0 \in \Omega$, then $  M \varphi (z_0) \geq f(z_0, \varphi(z_0))$ (resp. $  M \varphi (z_0) \leq f(z_0, \varphi(z_0) )$)

\medskip

{\rm (b)} Let $u \in  PSH (\Omega) \cap C(\Omega)$. Then $u$ is a subsolution if and only if $u$ satisfies the condition in {\rm (a)} with $\varphi \in C^{2} (\Omega) \cap PSH(\Omega)$ replaced by $\varphi \in C^2 (\Omega)$
\end{prop}

\begin{proof}
Note that if $p \in T^{+}$ and $M_{\C} (p) >0$, then there exists some $\epsilon_0$ such that for all $\epsilon < \epsilon_0$, $dd^c ( p  - \frac{\epsilon}{2} \abs{z}^2 )$ is strictly positive. If $M_{\C} (p) =0 $, then $M_{\C} (p) < f(z, u (z))$ for all $z$ as $f$ is non-negative.

With the above observation, (a) can be proved by mimicking the argument for Proposition 2.4 in \cite{CC}.

To see (b), one just needs to notice that we assume $u$ to be a psh. 
In this case, if $\varphi \in C^2 (\Omega)$ and $u - \varphi$ takes a local maximum at 
$z_0 \in \Omega$,  $dd^c \varphi $ has to be non-negative.
\end{proof}

We make the following definition for convenience:

\smallskip

\begin{defn}
\label{T2}
$u \in C(\Omega)$ is said to be $T_2$ at a point $z_0 \in \Omega$ if there exists a quadratic polynomial $\varphi$ such that:
\[
u (z) = \varphi(z) +  o ( \abs{z- z_0}^2 ) \quad \; \text{as} \; z \rightarrow z_0.
\]
In this case, we define $D^2 u$ to be $D^2 \varphi$ and $dd^c u$ to be $dd^c \varphi$. 
\end{defn}

It is clear that such a $\varphi$ is unique if exists, and hence $D^2 u, dd^c u$ are well-defined.

 \begin{prop} 
 \label{classical}
 Let $u $ be a continuous in $\Omega$.
\smallskip

{\rm (a)} if $u $ is a viscosity subsolution (resp. supersolution) of $M_{\C} (\cdot) = f (z, \cdot)$. then for any holomorphic function $h$ on $\Omega$, $M_{\C} ( u + \Re(h) ) \geq f(z, u (z))$ (resp. $M_{\C} ( u + \Re(h) ) \leq f(z, u (z))$) in the viscosity sense in $\Omega$.
 
  In particular, if $u$ is a viscosity subsolution (resp. supersolution) of $M_{\C} (\cdot) = f (z, \cdot)$, then  so is $u  -c$ (resp. $u + c$) for any positive constant $c$. 
 
 \smallskip
 
{\rm (b)} Let $u $ be a subsolution (resp. supersolution) for the equation $M_{\C} (\cdot) = f (z, \cdot)$ and $w$ be a continuous psh function that touches $u$ from above (resp. below) at some $z_0 \in \Omega$. Suppose that $w$ is $T_2$ at $z_0$. Then $M_{\C} (w) (z_0)  \geq f (z_0, u(z_0))$ (resp. $M_{\C} (w) (z_0) \leq f (z_0, u(z_0))$). 
 
In particular, if $u$ is viscosity subsolution (resp.supersolution) of $M_{\C} (u) = f (z, u)$ and $u$ is $T_2$ at a point $z_0 \in \Omega$, then $M_{\C} (u) (z_0)  \geq f (z_0, u(z_0))$ (resp. $M_{\C}u (z_0) \leq f (z_0, u(z_0))$).

\smallskip 

{\rm(c)} If $u \in C^2(\Omega) \cap PSH (\Omega)$, then $u$ is a viscosity subsolution (resp. supersolution) of $M_{\C} (u) = f(z,u)$ if and only if $M_{\C} (u) \geq f (z,u)$ (resp. $M_{\C} (u) \leq f (z,u)$) in classical sense. 
 \end{prop} 

\begin{proof}
With the observations mentioned in the proof of Proposition \ref{test condition}, to prove the above proposition, we just need to carry out the arguments for Lemma 2.5 and Corollary 2.6 in \cite{CC}.
\end{proof}

We end this section with the following convergence Proposition We refer to Proposition2.9 in \cite{CC} for a proof.

\begin{prop}
\label{conv uniform}
If $u_k$ be a sequence of viscosity subsolution (resp. supersolution) of $M_{\C} (\cdot) = f(z, \cdot)$ converging uniformly in compact subset of $\Omega$ to $u$. Then $M_{\C} (u) \geq f (z, u)$ (resp.$M_{\C} (u) \leq f (z, u)$) in $\Omega$ in viscosity sense.
\end{prop}

\section{Convex Envelopes and Jensen's Approximation}

We recall some results for convex envelopes. Our main reference is Ch.3 of \cite{CC}.

\begin{defn}
Let $w$ be a continuous function in an open bounded domain $\Omega$ such that  $\Omega \subset \subset B_r \subset B_{2r}$ and $w \geq 0$ on $\partial \Omega$. Let $w^{-}:= -\inf\{w,0\}$ and extend it to $B_{2r}$ by zero. Define:
\[\begin{split}
\Gamma_w (x) & : =\sup \{v (x)  : v \; \text{is convex in } B_{2r}, \; v \leq -w^{-} \;  \text{in } \Omega\}.
\end{split}\]
We call the set $ \{ w = \Gamma_{w}\}$ the contact set of $w$ in $\Omega$;
\end{defn}

\begin{rem}
Unless $w^{-}$ is identically zero, $\Gamma_w$ is strictly negative in the interior of $\Omega$.
\end{rem}

We recall some terminology in order to state an essential estimate -- the Alexandrov-Bakelman-Pucci (ABP) estimate.

A function P is called a paraboloid of opening $K$ if it is of the form
\[
P (x) = \frac{K}{2} \abs{x}^2 + l 
\]
where $l$ is an affine function.

\begin{defn}
Let $\Omega$ be a bounded domain in $\R^n$ and let $w \in C(\Omega)$. 
A function $w$ is said to be $K$-semi-concave in $\Omega$, if for any point $x \in \Omega$, there exists a paraboloid of opening $K$ touch $w$ from above in $\Omega$.

Similarly, one define $K$-semi-convexity.

$w$ is semi-concave (resp. semi-convex) if there exists a finite $K>0$ such that $w$ is $K$-semi-concave (resp. semi-convex).
\end{defn}

\begin{rem}
By Alexandrov theorem on second order differentiability (see \S 1.1 of [CC]), a semi-convex (semi-concave) functions are almost every $T_2$. 
\end{rem}

Now we shall state the ABP estimate

\begin{thm}
\label{ABP}
Let $w$ be continuous in $\Omega \subset B_r \subset B_{2r} \subset \R^n$ and $w \geq 0$ on $\partial \Omega$. Assume that $w$ is semi-concave. Then 

\smallskip
i) $\Gamma_w \in C^{1,1} (\overline{B_r})$;

\smallskip
ii) $\{w =\Gamma_w\} \subset \Omega$ unless $w^{-}$ is identically zero. 

\smallskip
iii) The image of the set $\{w = \Gamma_w\}$ under the normal mapping of $\Gamma_u$ contains a ball of radius $\sup{w^{-}} /2r$, that is, 
\[
B_{ \sup{w}^{-}/2r } \subset \nabla \Gamma_w (\{w = \Gamma_w\})
\]
\end{thm}

\begin{rem}
ABP estimates holds under weaker assumptions. For details, one may refer to \S3.1 of \cite{CC}.
\end{rem}

Following immediate consequence of \ref{ABP} will play key role in our proof of Theorem \ref{intro CP}.
\begin{cor}
\label{ABP gradient}
Let $w \in C(\overline{\Omega})$ be semi-concave and $E \subset \Omega$ be a set such that $\abs{\Omega \setminus E} = 0$. Suppose $w \geq 0$ on $\partial \Omega$ and $\min_{\Omega} w = -a, a > 0$. Then for any $\delta \in (0, a/2d), d =\diam{\Omega}$, there exists a point $x_0 \in E$ such that: 

\smallskip
i) $w (x_0)) = \Gamma_{w} (x_0)< 0$.

\smallskip
ii) $\Gamma_w$ is $T_2$ at $x_0$ and
\[
\abs{\nabla w} (x_0) = \abs{\nabla \Gamma_w} (x_0) < \delta , \quad  ( \det_{\R} D^2 \Gamma_{w} (x_0) )^{1/n} \geq \frac{\delta}{d}
\]
\end{cor}

\begin{proof}
We may assume $0 \in \Omega \subset B_{d}$. Let $H$ be the set consisting $T_2$-points of $\Gamma_u$ in $\Omega$. i) of Theorem \ref{ABP} implies 
\[
\abs{\Omega \setminus H} = 0.
\]

Fix $\delta \in (0, a/2r)$, then by ii) of Theorem\ref{ABP}, there exists a non-empty subset $A $ of the contact set $\{w = \Gamma_w\}$ such that
\[
\nabla \Gamma_{w} (A) = B_{\delta}
\]
and any point $x \in A \cap H$ satisfies:
\[
w (x) = \Gamma_{w} (x) < 0, \quad \abs{\nabla w}(x)  = \abs{\nabla \Gamma_w (x) } < \delta
\] 

Moreover, the $C^{1,1}$-regularity along with the area formula implies.
\begin{equation}
\label{integral}
\abs{B_{\delta}} \leq  \abs{\nabla \Gamma_{w} (A )} = \int_{A \cap E \cap H} \det D^2 \Gamma_w \; dx.
\end{equation}
Here we used the fact $\abs{A \setminus (E \cup H)} = 0$. 

Since $A \subset \Omega \subset B_{d}$, (\ref{integral}) implies $A \setminus (E \cup H)$ has positive measure and there exits $x_0 \in A \cap E\cap H$ such that
\[
\det D^2 \Gamma_w (x_0) \geq \frac{\delta^n}{d^n}.
\]
This completes the proof of the corollary.
\end{proof}

\begin{rem}
For the definition and properties of normal mapping of a convex function, one may refer to \cite{Gu}.
\end{rem}

We now adapt standard Jensen's approximation theorem to our setting.
Our main reference is Ch.5 of \cite{CC}. 

\begin{defn}
Let $u \in C(\overline{\Omega} ) \cap PSH (\Omega)$ and $M_{\C} (u)\geq f(x, u)$ in viscosity sense in $\Omega$. For $\epsilon > 0$, define
\[ \begin{split}
	& u^{\epsilon} (z_0):= \sup_{z \in \overline{\Omega} } \{ u(z) -\frac{1}{\epsilon}  \abs{z -z_0}^2 \} , \quad  z_0 \in \Omega ;\\
 &  f_{\epsilon} (z_0, t):= \inf  \{ f (z, t) : z \in B_{\tau} (x_0)  \cap \overline{\Omega} \} , \quad  \tau =( \epsilon \osc_{\Omega} u )^{1/2}
\end{split}
\] 

Similarly, one define $v_{\epsilon}, f^{\epsilon}$ for a supersolution.
\end{defn}

\begin{prop}
\label{Jensen}
Assume that $u \in C(\overline{\Omega} ) \cap PSH (\Omega)$ and $M_{\C} (u) \geq f(x, u).$ Then:

\smallskip
{\rm (a)} $u^{\epsilon} \in C^{0,1} (\overline{\Omega}) \cap PSH (\Omega)$ with Lipschitz constant smaller than $\frac{3}{\epsilon} \diam{\Omega}$.

\smallskip
{\rm (b)} $u^{\epsilon}$ decreases uniformly to $u$ in $\overline{\Omega}$.

\smallskip
{\rm (c)} $u^{\epsilon}$ is semi-convex in $\Omega$.

\smallskip
{\rm (d)} $f_{\epsilon}$ increases uniformly to $f$.

\smallskip

{\rm (e)} For any compact domain $U$ of $\Omega$, there exists a $\epsilon_0 $, depending on $\osc_{\overline{\Omega}} u$ and $\dist{U, \partial \Omega}$, such that, for any $\epsilon < 0$
\[
M_{\C} (u^{\epsilon}) \geq f_{\epsilon} (x ,u^{\epsilon}), \quad \text{ in } U
\]
in viscosity sense. 

Corresponding statements holds for $v_{\epsilon},  f^{\epsilon}$ of a supersolution.
\end{prop}

\begin{proof}
To prove (b), (c) and (e), one can carry out the argument on page 44 of \cite{CC} with obvious modification. (d) is clear. In (a), Lipschitz part is also proved in \cite{CC}, the PSH part follows from the change of variable $y = z-z_0$ and Choquet Lemma.
\end{proof}

\section{Dirichlet Problem}

First, we establish the comparison principle in Theorem \ref{intro CP}. The key idea of our proof originates from Sec.5 of \cite{CC}.

\medskip
\textit{Proof of Theorem \ref{intro CP}.} Without lose of generality, we may assume $0 \in \Omega$. Let $d = \diam{\Omega}$. Moreover, by replacing $v$ by $v + \delta$, we may assume $v > u$ on $\partial \Omega$. Theorem\ref{intro CP} will follow from this case by taking $\delta$ to be zero.

Argue by contradiction. Write $w = v - u$. Assume there exists $x_0 \in \Omega$ such that
\[
w (x_0 )  = -a < 0.
\]
Regularizing $v, u$ via Jensen's approximation, and write $w_{\epsilon} = v_{\epsilon} - u^{\epsilon}$. Since $w >0 $ on $\partial \Omega$ and $w_{\epsilon}$ increases uniformly to $w$, we may fix a compact subdomain $U$ of $\Omega$ such that $w_{\epsilon} \geq 0 $ in $\Omega \setminus U$  for all $\epsilon$ sufficiently small,

Fix any $\epsilon > 0$ small, denote $E_{\epsilon} \subset U$ the set consists of points on which $w_{\epsilon}, v_{\epsilon}, -u^{\epsilon}$ are $T_2$. By (b) of Proposition \ref{Jensen}, $\abs{U \setminus E_{\epsilon}}  = 0$ and $w_{\epsilon}$ is semi-concave. Apply Corollary \ref{ABP gradient}, we may choose $x_{\epsilon} \in E_{\epsilon}$ such that
\begin{equation}
\label{cp1}
w_{\epsilon} (x_{\epsilon} )= \Gamma_{w} (x_{\epsilon})< 0 , \quad \text{det}_{\R}^{1/2n} (D^2 \Gamma_{w_{\epsilon}}) (x_{\epsilon}) \geq \frac{a}{3d^2}
\end{equation}

By applying (e) of Proposition \ref{Jensen} with respect $U$, we obtain
\begin{equation}
\label{cp2}
M_{\C} (v^{\epsilon}) (x_{\epsilon}) \leq f^{\epsilon} (x_{\epsilon}, v_{\epsilon} (x_{\epsilon}) ).
\end{equation}
and
\begin{equation}
\label{cp3}
M_{\C} (u^{\epsilon}) (x_{\epsilon}) \geq f_{\epsilon} (x_{\epsilon}, u^{\epsilon} (x_{\epsilon})).
\end{equation}
 for all sufficiently small $\epsilon $. Here the expression are computed in the sense of (b) of Proposition \ref{classical}. 

Combine Eq.( \ref{cp1}, \ref{cp2}, \ref{cp3}) and use the Minkowski inequality of determinant and Lemma \ref{RC inequality}, along with the fact $\Gamma_{w_{\epsilon}} + u^{\epsilon}$ touches $v_{\epsilon}$ from below at $x_{\epsilon}$. we obtain:
\begin{equation}
\label{cp cd}
\begin{split}
f^{\epsilon} (x_{\epsilon},   v_{\epsilon} (x_{\epsilon}))   & \geq f_{\epsilon} (x_{\epsilon}, u^{\epsilon} (x_{\epsilon}) ) + \frac{a}{3d^2} \\ 
&  \geq f_{\epsilon} (x_{\epsilon}, v_{\epsilon} (x_{\epsilon}) )   + \frac{a}{3d^2} 
\end{split}
\end{equation}
In the last inequality, we have used the fact that $f(x, t)$ is non-decreasing in $t$.

Since $a/3d^2$ is independent of $\epsilon$ and $f^{\epsilon} , f_{\epsilon}$ converges uniform to $f$, Eq.\ref{cp cd} leads to a contradiction when $\epsilon$ is taken sufficiently small. $\Box$

\begin{rem}
i) The continuity of $u,v $ is not essential to the proof. Semi-continuity with corresponding boundary condition is sufficient.

ii) This proof can be applied to more general operator with suitable structure conditions. 
\end{rem}

\medskip

Now, we shall apply Perron Method to prove Theorem \ref{intro solve}. Our argument, learned from Prof. Ovidiu Savin, differs from standard argument in \cite{IL}.

The key part is the following lemma:

\begin{lem}
\label{sub replace}
Under assumption of Theorem \ref{intro solve}. For each subsolution $u$ with $u |_{\partial \Omega} \leq g$, there exists $\tilde{u} \in C(\overline{\Omega}) \cap PSH (\Omega)$ such that 

\smallskip
i) $\tilde{u}$ is a viscosity subsolution of $M (\cdot ) = f(x, \cdot)  $ in $\Omega$.

\smallskip
ii) $\tilde{u} \geq u$ in $\overline{\Omega}$ and $\tilde{u} |_{\partial \Omega}= g$.

\smallskip 
iii) The modulus of continuity $\omega_{\tilde{u}}$ of $\tilde{u}$ satisfies: for all $z_1, z_2 \in \overline{\Omega}$
\[
\omega_{\tilde{u}} (\abs{z_1 - z_2}) \leq (3d \osc_{V} f ) \abs{z_1 - z_2} + d^2  \omega_f (\abs{z_1 - z_2} + \omega (\abs{z_1 -z_2})).
\]
where $\omega = \max\{ \omega_{\underline{u}}, \omega_{\overline{u}}\}$ and $\omega_{f}$ is the modulus of continuity of $f^{1/n}$ in the region 
\[
V = \overline{\Omega} \times [ -M , M], \quad M = \max\{ \norm{g}_{\infty} + d^2, \norm{\underline{u}}_{\infty} + d^2\}.
\]

\end{lem}

\begin{proof}
The proof is similar to that of Jensen's approximation (see \S5.1 [CC]). However this lemma seems not standard and well-known, we shall include details for readers' convenience.

Without lose of generality, we shall assume $0 \in \Omega$ and $d = \diam{\Omega}$.  By taking $\sup\{u, \tilde{u}\}$, we may assume $u |_{\partial \Omega}=  g$ and $u \geq \tilde{u}$ in $\partial \Omega$.

Define $\tilde{u}$ as following:
\[
\tilde{u}(z_0) := \sup_{y \in \overline{\Omega}} \{ \max [ u(y) -\omega (\abs{y -z_0})  + \frac{\omega_f (\abs{y -z_0})}{2} (\abs{z_0}^2 - d^2)  , \;  \underline{u} (z_0) ] \}.
\]
Write $z_0^*$ for the point where the supreme  occur.  We shall show $\tilde{u}$ satisfies desired properties.

We prove iii) first and $u \in PSH (\Omega)$ follows from Choquet lemma. The key point is that (optimal) modulus of continuity function is sub-additive.  Fix $z_1, z_2 \in \overline{\Omega}$. By additivity,
\[
\abs{\omega (\abs{z_1 - y}) - \omega (\abs{z_2 -y}) } \leq \omega (\abs{z_2 - z_1}), \; \forall y \in \overline{\Omega}
\]
Similarly holds for $\omega_{f}$

Therefore,
\[\begin{split}
& \omega_{f} (\abs{z_1 -y})( \abs{z_1}^2 - d^2 ) - \omega_f(\abs{z_2- y}) (\abs{z_2}^2 - d^2) \\
 & =\omega_f(\abs{z_1 - y}) (\abs{z_1}^2 - \abs{z_2}^2) + (\omega_f (\abs{z_1-y}) -\omega_f (\abs{z_2- y}) ) (\abs{z_2}^2 - d^2) \\
& \geq - 3d ( \osc_{V} f ) \; \abs{z_1 - z_2} -  d^2\omega_f(\abs{z_1 -z_2}) 
\end{split}\]

Combine all these, we obtain
\[\begin{split}
\tilde{u} (z_1) - \tilde{u} (z_2) \geq  - (3d \osc_{V} f ) \abs{z_1 - z_2}- d^2  \omega (\abs{z_1 - z_2})  - \omega (\abs{z_1 - z_2})
\end{split}\]

Since $z_1, z_2$ are chosen arbitrarily
\[
\omega_{\tilde{u}} ( \abs{z_1 - z_2}) \leq (3d \osc_{V} f ) \abs{z_1 - z_2} + d^2  \omega_f (z_1 - z_2 ) + \omega (\abs{z_1 -z_2})
\]
This proves iii).

\smallskip

To prove i). Let $p\in T^{+} (n), h \in H(n)$ and $P = p + \Re(h)$ touches $\tilde{u}$ from above at $x_0 \in \Omega$. If $\tilde{u} (x_0) = \underline{u} (x_0)$, then 
\[
M_{\C} (P) \geq f (x_0, \underline{u} (x_0)),
\]
because $\underline{u}$ is a subsolution.

If $\tilde{u} (x_0) > \underline{u} (x_0)$, then the quadratic polynomial
\[
Q (z ) := P (z + x_0 - x_0^*) + \omega (\abs{x_0 - x_0^*}) - \frac{\omega_f(\abs{x_0-x_0^*})}{2} (\abs{z + x_0 - x_0^*}^2 -d^2)
\] 
touches $u$ at $x_0^*$. 

Claim $x_0^* \in \Omega$. Suppose otherwise $x_0^* \in \partial \Omega$ ,then
\[\begin{split}
 \tilde{u} (x_0) &= u (x_0^*) - \omega (\abs{x_0^* - x_0}) + \frac{\omega_f (\abs{x_0^* - x_0})}{2} (\abs{x_0}^2 - d^2) \\
& \leq \overline{u} (x_0^*) - \omega (\abs{x_0^* - x_0}) \leq  \overline{u} (x_0^* ) = \underline{u} (x_0^*).
\end{split}\]
Here we used the fact that $u$ is subharmonic and $u = \overline{u}$ on $\partial \Omega$. But this contradicts to the assumption that $\tilde{u} (x_0) > \underline{u} (x_0)$.

Now, recall $u$ is a subsolution, thus
\[
M_{\C} (Q)   \geq f (x_0^* , u (x_0^*) ) = f (x_0^* , u (x_0^*)) 
\]
Apply Minkowski inequality of determinant, we obtain 
\[\begin{split}
M_{\C}^{1/n} (P) & \geq f^{1/n} (x_0^*, u (x_0^*) ) + \omega (\abs{x_0^* -x_0})  \\
& \geq f^{1/n} (x_0^*, \tilde{u} (x_0) ) + \omega (\abs{x_0^* -x_0})  \geq f^{1/n} (x_0, \tilde{u} (x_0) )
\end{split}\]
Here we have used the fact
\[
\tilde{u} (x_0 ) =  u (x_0^* ) - \omega (\abs{x_0^* - x_0}) + \frac{\omega (\abs{x_0^* - x_0})}{2} (\abs{x_0}^2 - d^2) \leq u (x_0^*).
\]
Thus, we have shown $\tilde{u}$ is a subsolution.

\smallskip 
To see ii),  same as in the previous step,  if $z_0 \in \partial \Omega$, then for any $y \in \overline{\Omega}$
\[\begin{split}
&u (y) - \omega (\abs{ y- z_0}) + \frac{\omega (\abs{y - z_0})}{2} (\abs{z_0}^2 - d^2) \\
& \leq \overline{u} (y) - \omega (\abs{z_0 - y}) \leq \overline{u}(z_0) = \underline{u} (z_0).
\end{split}\]
Hence
\[
\tilde{u} |_{\partial \Omega} = \underline{u } |_{\partial \Omega} = g .
\]
That $\tilde{u} \geq u$ in $\Omega$ is obvious.

\smallskip

This completes the proof.
\end{proof}

\medskip

\textit{Proof of Theorem \ref{intro solve}.} Define
\[
\subsoln: = \{ v : M_{\C} (v) \geq f(x, v) \text{ in viscosity sense in } \Omega, \; v|_{\partial \Omega} \leq g  \}.
\]
and
\[
u = \sup \{v : v \in \subsoln\}.
\]

Consider the following sub-family of $\subsoln$
\[
\tilde{\subsoln}:= \{\tilde{v} : v \in \subsoln\}.
\]

By the Lemma \ref{sub replace}, $\tilde{\subsoln}$ is a equi-continuous subset of $\subsoln$ and
\[
u =  \sup \{\tilde{v} : \tilde{v} \in \tilde{\subsoln}\}.
\]  

By Arzelà--Ascoli, $u$ is uniform limit of a sequence of subsolution. Hence, by recalling Proposition \ref{conv uniform}, $u \in C (\Omega)$ is again a subsolution.

To finish the proof, we only left to show that $u$ is also a supersolution. Argue by contradiction, there exists a $p \in T^{+} (n), h \in H(n)$ such that $p + \Re(h)$ touches $u$ from below at $z_0 \in \Omega$, but $M_{\C} (p) > f (z, u(z_0))$. Then, taking a small ball $B_{r}$ of $z_0$, and define
\[
\psi := p  + \Re (h) - \frac{\epsilon}{2} (\abs{z-z_0^2} - r^2).
\]
For $r, \epsilon$ small enough, we have $\psi$ is Psh,
\[
M_{\C} (\psi)  >  f (z, u (z)), \quad \forall  z \in B_{r} (z_0).
\] 
and 
\[
\psi (z_0) > u (z_0), \quad \psi|_{\partial B_{r}} \leq u |_{\partial B_{r}}.
\]

Then define
\[
\hat{u}:= \begin{cases}
 \max \{\psi , u \} & x \in \overline{B_r} (z_0) \\
  u     &   x \in \Omega \setminus B_r (z_0).
\end{cases}
\]

It is easy to check that $\hat{u}$ is again a subsolution but $\hat{u} (z_0) > u (z_0)$. This contradicts the maximality of $u$.

The proof is then complete.
$\Box$

\medskip

Besides solvability, Lemma \ref{sub replace} and the proof of Theorem \ref{intro solve} yields 

\begin{cor}
\label{mod}
Let $\Omega$ be a strictly pseudoconvex domain and $\phi \in C(\overline{\Omega})$ and $\psi (u) \in C(\R)$ non-decreasing. Suppose $M_{\C} (\cdot) = \phi (x) \psi (u)$ admits an subsolution $\underline{u}$ with modulus of continuity $\omega_{\underline{u}}$. Then the modulus of continuity $\omega_u$ unique viscosity solution $u$ satisfies
\[
\omega_{u} \leq C (\max\{\omega_{\phi}, \omega_{\underline{u}}\})
\]
where $C$ depends on $n ,\diam{\Omega}, \norm{\phi}_{L^{\infty} (\Omega)}$ and the $L^{\infty}$ norm of $\psi$ on the interval $[ \min_{\Omega}\underline{u}, \max_{\Omega} \underline{u}]$.

In particular, if $\underline{u}$ and $\phi$ are $\alpha$-H\"older in $\overline{\Omega}$, then $u$ is $\alpha$-H\"older with 
\[
\norm{u}_{\alpha, \overline{\Omega}} \leq C \max\{\norm{\underline{u}}_{\alpha, \overline{\Omega}}, \norm{\phi}_{\alpha, \overline{\Omega}} \}
\]
\end{cor}

\begin{proof}
Since $\Omega$ is strictly pseudoconvex, in particular of $C^2$-boundary. By standard harmonic analysis, there exists a harmonic function $h$ such that $h |_{\partial \Omega} = \underline{u}$ and 
\[
\omega_h \leq C \omega_{\underline{u}}
\]
where $C$ only depends on $n , \diam{\Omega}$.

Apply Lemma \ref{sub replace} to the viscosity solution $u$, then $\tilde{u} \geq u$. By tracking the proof of Lemma \ref{sub replace}, the specific form $f (x, t) = \phi (x) \psi (u)$ allows one to conclude that $\omega_{\tilde{u}}$ has required modulus of continuity. But $u \geq \tilde{u}$ by comparison principle, hence $u = \tilde{u}$. And the corollary follows. 
\end{proof}

\section{Relations with Pluripotential Solutions}

In this section, we assume that $\Omega$ is a strictly bounded pseudoconvex domain. We use the following normalization of Lebesgue measure:
\[
d\lambda  = \frac{1}{n!} ( \sum_i \ci dz^i \wedge d\bar{z^i}).
\]
Hence 
\[
(dd^c u) = \det( 2 u_{\bar{k} j } ) \; d\lambda.
\]

\begin{prop}
\label{p to v}
Let $f \in C(\Omega)$ be non-negative and $u\in C(\Omega)$. Then $(dd^c u )= f(z) \; d \lambda$ in pluripotential-potential sense implies that $M_{\C} (u) = f(z)$ in viscosity sense.
\end{prop}

\begin{proof}
Let $u$ be a pluripotential solution of $M_{\C} (u) = f(z)$. To see that $u$ is a viscosity subsolution, one may carry out argument of proposition \cite{EGZ2} words by words. Now, it suffices to show that $u$ is a viscosity supersolution.

Let $p \in T^{+} (n), h \in H(n)$ and $p+  \Re(h) $ touches $u$ from below  at some $z_0 \in \Omega$. We need to verify that $M_{\C} ( p) \leq f(z_0)$. Without lose of generality, we assume $z_0 = 0$.

Suppose on the contrary, $M_{\C} (p) > f(0) \geq 0$, By continuity of $f$, there exists some $r_0$ and $\epsilon_0$ such that  
\[
dd^c ( p  - \frac{\epsilon_0}{2} \abs{z}^2) > 0;\quad M_{\C} (p - \frac{\epsilon_0}{2} \abs{z}^2 ) (z) > f(z), \;\forall z \in \overline{B_{r_0}} 
\]

Choose $0 <\delta <\frac{\epsilon_0 } {2} r_0^2 $, then 
 \[ \begin{split}
&  p + \Re(h )  -  \frac{\epsilon } {2} \abs{r_0}^2 + \delta <  u   , \; \text{on} \; \partial B_{r_0} \\ 
& p (0) + \Re(h) (0)- \frac{\epsilon } {2} \abs{z}^2_{z = 0} + \delta   > u (0) 
\end{split}
\]
But this contradicts the pluripotential comparison principle. Thus a pluripotential solution $u$ is also a viscosity supersolution. 
\end{proof}

\medskip

Theorem \ref{intro eqv} then follows easily from Proposition \ref{p to v}:

\medskip

\textit{Proof
of Theorem \ref{intro eqv}}: Let $u \in C(\overline{\Omega})$ satisfy $M_{\C} (\cdot) = f$ in viscosity sense. Solve the Dirichlet problem (\ref{intro DP}) with data $f, g = u|_{\partial \Omega}$ in pluripotential sense. Denote the unique solution by $\tilde{u}$. By Proposition\ref{p to v}. $\tilde{u}$ is also a viscosity solution of the Dirichlet problem. The viscosity uniqueness forces $\tilde{u} = u$. $\Box$

\medskip

Finally, we prove the corollary \ref{intro cor}

\medskip

\textit{Proof of Corollary \ref{intro cor}:} By Theorem \ref{intro eqv}, it suffices to solve the Dirichlet problem (\ref{intro DP}) in the viscosity sense. 

First consider the case when $g \in C^{\infty} (\overline{\Omega})$. Let $\rho \in C^2 (\overline{\Omega})$ be the exhaustion function of $\Omega$, then $A \rho + g$ is a subsolution with boundary value $g$ for $A$ sufficiently large. And the $C^2$-boundary allows the existence of harmonic functions for arbitrary given continuous boundary data. Therefore following Theorem (\ref{intro DP}), we obtain the existence and uniqueness.

The general case when $g$ is continuous follows from the above special case by applying a standard approximation procedure based on Theorem \ref{intro CP} and Proposition \ref{conv uniform}. 		$\Box$

\begin{rem}
We use here the
fact that strictly pseudoconvex domains have $C^2$ exhaustion functions.
\end{rem}

\section{An ABP-type of $L^{\infty}$-estimate}

Let $\Omega$ be a bounded pseudoconvex domain and $u$ solves
\[
\begin{cases}
( dd^c u )^n = f \in C(\overline{\Omega}) & \text{in } \Omega\\
 u = 0 & \text{on } \partial \Omega
\end{cases}.
\]

It was originally established by Cheng and Yau (see \cite{B}, p. 75) that
\[
\norm{u}_{L^{\infty} (\Omega)} \leq c_n \diam{\Omega} \norm{f}_{L^2 (\Omega)}^{1/n}.
\]
The Cheng-Yau argument was made precise by Cegrell and Persson \cite{CL}. Moreover, it is pointed out in Bedford \cite{BT1} that 
\[
\norm{u}_{L^{\infty} (\Omega)} \leq c_n \diam{\Omega} \norm{f \chi_{\{ u =\Gamma_u\} } }_{L^2}^{1/n}.
\]
when $\Omega$ is convex. 
Using the viscosity techniques developed in this paper, we can make the proof in \cite{BT1} more transparent:

\begin{lem}
\label{eq conv env}
Under the hypotheses of Theorem \ref{C0e ABP}, 
$\Gamma_u$ is a viscosity supersolution of the real Monge-Amp\'ere equation 
\[ 
M_{\R} ( \cdot ) = f^2 \chi_{u = \Gamma_u}. 
\]
and $\Gamma_u = 0$ on $\partial B_{2r}$. 
(see \cite{Gu} for the definition of viscosity solutions for
real Monge-Amp\`ere equations)
\end{lem}

\begin{proof}
The fact that $\Gamma_{u} = 0$ on $\partial B_{2r}$ is trivial.

Let $z_0 \in B_{2r} \setminus \{ u  =\Gamma_u\}$. It is shown in p. 27 of \cite{CC} that there exists an open line segment $L$ through $z_0$ on which $\Gamma_u$ is affine. Hence, if $P$ is a convex polynomial touching $\Gamma_u$ from below at $z_0$, then $P$ is affine on $L$, hence $\det D^2 P  = 0$.

Let $z_0 \in B_{2r} \in \{ u  =\Gamma_u\}$ and $P$ be a convex polynomial touching $\Gamma_u$ from below at $z_0$. Then $P$ touches $u$ from below at $z_0$. 
Since $u$ is a viscosity supersolution of $M_{\C} (\cdot) = f $ and we can apply
Lemma \ref{RC inequality}, we obtain
\[
M_{\R}^{1/2n} (P) \leq M_{\C}^{1/n} (P) \leq f^{1/n} (z_0).
\]
\end{proof}

\medskip

\textit{Proof of Theorem
\ref{C0e ABP}:} By the standard theory of real Monge-Amp\'ere equation (see Proposition
1.7.1 of \cite{Gu} for example), Lemma \ref{eq conv env} implies that  
\[
\abs{\nabla \Gamma_u  (U)} \leq \int_{U} f^2 \; dx 
\]
for all open sets $U \subset B_{2r}$. Thus, by Alexandrov's maximum principle (Theorem
1.4.1 of \cite{Gu}), we can conclude that
\begin{equation}
\sup_{B_{2r} } -\Gamma_u \leq C(n) r  \norm{ f \cdot \chi_{\{  u= \Gamma_u    \}}}^{1/n}_{L^{2} (B_{2r})}.
\end{equation}
Since $u^{-} \leq - \Gamma_u$ and $\{u = \Gamma_u\} \subset \Omega$ unless $u^- =0$ (see Theorem 3.6 in \cite{CC}), the desired estimate follows. 
Theorem \ref{C0e ABP} is proved. $\Box$

\begin{rem}
I) Kolodziej \cite{Kol2} has shown that for any $p > 1$,  $(dd^c u )^n = f \in L^p (\Omega)$ and $u \geq  0 $ on $\partial \Omega$ implies 
\[
\sup_{\Omega} u^{-} \leq C (p, n , \diam{\Omega} ) \norm{f}_{L^{p} (\Omega)}^{1/n}.  
\]

It is not clear whether $\norm{f \chi_{ \{u= \Gamma_u \} } }_{L^2}$ can be controlled by $\norm{f}_{L^p}$ with $1 < p < 2 $, or vice versa.

ii) If $\Omega$ is convex, then Lemma \ref{eq conv env} holds for the convex envelope of $u$ in $\Omega$.
\end{rem}

Along the same line of proof as in \cite{CL}, Theorem
\ref{C0e ABP} implies the following stability theorem:

\begin{cor}
Let $f_1, f_2 \in C (\overline{\Omega})$ be non-negative and let $u_1,u_2 \in C(\overline{\Omega})$ be pluripotential-potential solutions of $M_{\C} (\cdot) = f_1 $ (resp. $M_{\C} (\cdot) = f_2$). Then
\begin{equation}
\label{stability}
\norm{u_1 - u_2}_{L^{\infty} (\Omega) } \leq \norm{u_1 - u_2}_{L^{{\infty} } ( \partial \Omega) }  +  C(n) \diam{\Omega}  \norm{ (f_1 - f_2 ) \cdot \chi_{\{ u =\Gamma_u\}}}_{L^2(\Omega)}^{1/n}.
\end{equation}
\end{cor}

\bibliographystyle{alpha}
\bibliography{vscma}

\end{document}